\documentclass[12pt,oneside,reqno]{amsart}

\usepackage{hyperref}

\usepackage[headsep=30pt,footskip=30pt,twoside=false,bottom=2.5cm]{geometry}

\usepackage[all]{xy}
\usepackage{tikz-cd}
\xymatrixcolsep{2cm}

\numberwithin{equation}{subsection}

\allowdisplaybreaks[1]

\usepackage{etoolbox}

\makeatletter

\patchcmd{\thesubsection}{\arabic}{\Alph}{}{}

\patchcmd{\@seccntformat}{\@secnumfont}{%
  \@secnumfont\expandafter\protect\csname format#1\endcsname}{}{}

\patchcmd{\@startsection}{\@afterindenttrue}{\@afterindentfalse}{}{}

\patchcmd{\subsection}{-.5em}{.3\linespacing}{}{}

\makeatother

\theoremstyle{plain}

\newtheorem{theorem}{Theorem}[section]

\newtheorem{proposition}[theorem]{Proposition}

\newtheorem{lemma}[theorem]{Lemma}

\newtheorem{corollary}[theorem]{Corollary}

\theoremstyle{remark}

\newtheorem{remark}[theorem]{Remark}

\newcommand{\sr}{\ensuremath{\mathrm{s}}}

\newcommand{\tg}{\ensuremath{\mathrm{t}}}

\newcommand{\Repr}[2][]{\ensuremath{\mathbf{Rep}_{#1}(#2)}}

\newcommand{\Hom}[3][]{\ensuremath{\mathrm{Hom}_{#1} (#2, #3)}}

\newcommand{\restrict}[2]{\ensuremath{#1\vert_{#2}}}

\newcommand{\End}[2][]{\ensuremath{\mathrm{End}_{#1} (#2)}}

\newcommand{\Limp}{\ensuremath{\Rightarrow}}

\newcommand{\id}[1]{\ensuremath{\mathbf{1}_{#1}}}

\renewcommand{\dim}[2][]{\ensuremath{\mathrm{dim}_{#1}(#2)}}

\newcommand{\N}{\ensuremath{\mathbf{N}}}

\newcommand{\Q}{\ensuremath{\mathbf{Q}}}

\newcommand{\R}{\ensuremath{\mathbf{R}}}

\newcommand{\set}[1]{\ensuremath{\{ #1 \}}}

\newcommand{\suchthat}{\ensuremath{\, \vert \,}}

\newcommand{\C}{\ensuremath{\mathbf{C}}}

\newcommand{\tr}[1]{\ensuremath{\mathrm{Tr}(#1)}}

\newcommand{\card}[1]{\ensuremath{\mathrm{card} (#1)}}

\newcommand{\pair}[2]{\ensuremath{\langle #1, #2 \rangle}}

\newcommand{\Lie}[1]{\ensuremath{\mathrm{Lie}(#1)}}

\newcommand{\norm}[1]{\ensuremath{\lVert#1\rVert}}

\newcommand{\Aut}[2][]{\ensuremath{\mathrm{Aut}_{#1} (#2)}}

\newcommand{\units}[1]{\ensuremath{#1^\times}}

\renewcommand{\bar}[1]{\ensuremath{\overline{#1}}}

\newcommand{\tang}[2][]{\ensuremath{\mathrm{T}_{#1}(#2)}}

\DeclareMathOperator{\diff}{d}

\begin{document}

\title[Properness of moment maps for representations of quivers]
{Properness of moment maps for representations of quivers}

\author[Pradeep Das]{Pradeep Das}

\email{pradeepdas@hri.res.in}

\address{Harish-Chandra Research Institute \\ HBNI \\ Chhatnag Road \\ Jhunsi \\ Allahabad 211~019 \\ India}



\begin{abstract}
  In this note, we give a necessary and sufficient condition for the
  properness of moment maps for representations of quivers.  We show
  that the moment map for the representations of a quiver is proper
  if and only if the quiver is acyclic, that is, it does not
  contain any oriented cycle.
\end{abstract}

\maketitle

\section*{Introduction}
\label{sec:introduction}

Moment maps play an important role in the study of the geometry of
moduli spaces and quotient spaces.  For example, the Morse theory of
the norm-square of the moment map for a symplectic action of a compact
Lie group on a symplectic manifold determines the Morse stratification
of the manifold \cite{FK}.  Moment maps are the key ingredients of
Marsden-Weinstein reduced spaces, which include, as important examples, several kinds of moduli spaces.  In \cite{SL}, Sjamaar and
Lerman have shown that singular Marsden-Weinstein reduced spaces are
symplectic stratified spaces, without any assumption on the moment
map.  In \cite{RS}, Sjamaar shows that the reduced space of a
reductive group action on a K\"ahler manifold is a complex space and
exhibits a K\"ahler stratification if the moment map is admissible,
where by admissibility we mean that for every point of the manifold,
the path of steepest descent through that point is contained in a
compact set.  He further shows that there exists a positive
holomorphic line bundle (in the sense of Grauert) on this reduced
space if the moment map is proper, thus proving that the reduced space
is a complex projective variety.  A proper moment map is always
admissible.

Several authors have studied the moment maps for the representations
of quivers.  In \cite{WCB}, Crawley-Boevey considers algebraic moment
maps on the cotangent bundle of the representation space of a quiver,
and has given a criterion for their flatness, and a stratification of
their Marsden-Weinstein reductions.  Using gradient flow lines for the
norm-square of the moment map, Wilkin, in \cite{GW}, has obtained a
Hecke correspondence which is used in the construction of the
generators of Kac-Moody algebras.  Harada and Wilkin \cite{HW} have
studied the Morse theory of the norm-square of the moment map in
detail.  They have shown that the moment map for the representations
of quivers is indeed admissible.  In \cite[Proposition 3.1]{HP}, Hausel and Proudfoot give a
sufficient condition for the properness of the moment map for quiver
representations.

In \cite{PMR}, the authors have studied some holomorphic and
symplectic aspects of representations of finite quivers.  Using the
formalism of that paper, we show that the sufficient condition for
the properness of the moment map given by Hausel and Proudfoot is
also necessary.  We also give another proof of the sufficiency.

The first section is devoted to the preliminaries required to prove
the main results, namely Proposition \ref{prop1} and Proposition
\ref{prop2}, of this note.  The details of the contents of Section
\ref{sec1} can be found in \cite{PMR}.

\section{Preliminaries}
\label{sec1}

\subsection{Quivers and their representations}
\label{sec:1.1}

A \emph{quiver} $Q$ is a quadruple $(Q_0,Q_1,\sr,\tg)$, where $Q_0$
and $Q_1$ are sets, and $\sr: Q_1\to Q_0$, and $\tg: Q_1\to Q_0$ are
functions.  The elements of $Q_0$(resp., $Q_1$) are called the
\emph{vertices}(resp., \emph{arrows}) of $Q$.  For any arrow $\alpha$
of $Q$, the vertex $\sr(\alpha)$(resp., $\tg(\alpha)$) is called the
\emph{source}(resp., \emph{target}) of $\alpha$.  If $\sr(\alpha)=a$
and $\tg(\alpha)=b$, then we say that $\alpha$ is an arrow
\emph{from} $a$ \emph{to} $b$, and write $\alpha:a\to b$.  The quiver
$Q$ is said to be \emph{vertex-finite}(resp. \emph{arrow-finite}) if
the set $Q_0$(resp.,$Q_1$) is finite, and \emph{finite} if it is both
vertex-finite and arrow-finite.  The quiver
$(\emptyset,\emptyset,\sr,\tg)$, where $\sr$ and $\tg$ are the empty
functions, is called the \emph{empty} quiver.  We say that a quiver
$Q$ is \emph{non-empty} if it is not equal to the empty quiver, or
equivalently, if the set $Q_0$ of its vertices is non-empty.

A \emph{path of length $l \geq 1$ from $a$ to $b$} in $Q$ is a
sequence $\set{\alpha_k}_{k = 1}^{l}$, where $\alpha_k \in Q_1$ for
all $1 \leq k \leq l$, with
$\sr(\alpha_1) = a, ~\tg(\alpha_k) = \sr(\alpha_{k+1})$ for each
$1 \leq k < l$, and finally $\tg(\alpha_l) = b$.  Such a path is
denoted briefly by $\alpha_1\alpha_2 \cdots \alpha_l$, and may be
visualised as follows
$$a = a_0
\overset{ \alpha_1} \longrightarrow  a_1
\overset{ \alpha_2} \longrightarrow  a_2
\longrightarrow \cdots
\overset{ \alpha_l} \longrightarrow  a_l
=b.$$
A path of length $l \geq 1$ is called a \emph{cycle} whenever its
source and target coincide.  A cycle of length $1$ is called a
\emph{loop}.  A quiver is called \emph{acyclic} if it contains no
cycles.

Let $k$ be a field.  A \emph{representation of $Q$ over $k$} is a pair
$(V,\rho)$, where $V=(V_a)_{a\in Q_0}$ is a family of
finite-dimensional $k$-vector spaces, and
$\rho=(\rho_\alpha)_{\alpha\in Q_1}$ is a family of $k$-linear maps
$\rho_\alpha : V_{\sr(\alpha)} \to V_{\tg(\alpha)}$.  For every
representation $(V,\rho)$, the element
\begin{equation*}
  \label{eq:68}
  \dim{V,\rho} = (\dim[k]{V_a})_{a\in Q_0}
\end{equation*}
of $\N^{Q_0}$ is called the \emph{dimension vector} of $(V,\rho)$.
If $(V,\rho)$ and $(W,\sigma)$ are two representations of $Q$, then a
\emph{morphism from $(V,\rho)$ to $(W,\sigma)$} is a family
$f=(f_a)_{a\in Q_0}$ of $k$-linear maps $f_a:V_a\to W_a$, such that
$f_{\tg(\alpha)} \circ \rho_\alpha =
 \sigma_\alpha \circ f_{\sr(\alpha)}$, for each $\alpha\in Q_1$.  If
$(V,\rho)$, $(W,\sigma)$, and $(X,\tau)$ are three representations
of $Q$, $f$ a morphism from $(V,\rho)$ to $(W,\sigma)$, and $g$ a
morphism from $(W,\sigma)$ to $(X,\tau)$, then the \emph{composite} of
$f$ and $g$ is the family $g\circ f$ defined by
$g\circ f = (g_a\circ f_a)_{a\in Q_0}$.  The composite of $f$ and $g$
is a morphism from $(V,\rho)$ to $(X,\tau)$.  We thus get a category
$\Repr[k]{Q}$, whose objects are representations of $Q$ over $k$, and
whose morphisms are defined as above.

%A \emph{subrepresentation} of a representation $(V,\rho)$ of $Q$ is a
%subobject of $(V,\rho)$ in the category $\Repr[k]{Q}$.  

\subsection{The group action}

Let $Q$ be a non-empty finite quiver.  From now onward we will consider only complex representations of $Q$. Let $d=(d_a)_{a\in Q_0}$
be a non-zero element of $\N^{Q_0}$, and fix a family
$V=(V_a)_{a\in Q_0}$ of $\C$-vector spaces, such that
$\dim[\C]{V_a}=d_a$ for all $a\in Q_0$.

Let $\mathcal{A}$ denote the finite-dimensional $\C$-vector space
$\bigoplus_{\alpha\in Q_1}\Hom[\C]{V_{\sr(\alpha)}}{V_{\tg(\alpha)}}$.
Then, for each $\rho \in \mathcal{A}$, we have a representation
$(V,\rho)$ of $Q$.  conversely, for every representation $(W,\sigma)$
of $Q$, such that $\dim{W,\sigma}=d$, there exists an element $\rho$
of $\mathcal{A}$, such that the representations $(V,\rho)$ and
$(W,\sigma)$ are isomorphic.  We give the vector space $\mathcal{A}$,
usually called the representation space of $Q$ with dimension vector
$d$, the usual topology, and the usual structure of a complex
manifold.

Let $G$ be the complex Lie group $\prod_{a\in Q_0}\Aut[\C]{V_a}$.
There is a canonical holomorphic linear right action
$(\rho,g)\mapsto \rho g$ of $G$ on $\mathcal{A}$, which is defined by
\begin{equation*}
  \label{eq:231}
  (\rho g)_\alpha =
  g_{\tg(\alpha)}^{-1} \circ \rho_\alpha \circ g_{\sr(\alpha)}
\end{equation*}
for all $\rho\in \mathcal{A}$, $g\in G$, and $\alpha\in Q_1$.  For all
$\rho,\sigma \in \mathcal{A}$ and $g\in G$, we have $\sigma = \rho g$
if and only if $g$ is an isomorphism of representations of $Q$, from
$(V,\sigma)$ to $(V,\rho)$.  In other words, two points $\rho$ and
$\sigma$ of $\mathcal{A}$ lie on the same orbit of $G$ if and only if
the representations $(V,\rho)$ and $(V,\sigma)$ of $Q$ are isomorphic.
Thus, the map which takes every point $\rho$ of $\mathcal{A}$ to the
representation $(V,\rho)$ induces a bijection from the quotient set
$\mathcal{A}/G$ onto the set of isomorphism classes of representations
$(W,\sigma)$ of $Q$, such that $\dim{W,\sigma} = d$.

Denote by $H$ the central complex Lie subgroup of $G$ consisting of
all elements of the form $ce$, as $c$ runs over $\units{\C}$, where
$e=(\id{V_a})_{a\in Q_0}$ is the identity element of $G$.  Let
$\bar{G}$ denote the complex Lie group $H\backslash G$,
$\pi : G \to \bar{G}$ the canonical projection.  Note that the action
of $G$ on $\mathcal{A}$ induces a holomorphic right action of
$\bar{G}$ on $\mathcal{A}$.

The Lie algebra $\Lie{G}$ of $G$ is the direct sum Lie algebra
$\bigoplus_{a\in Q_0} \End[\C]{V_a}$, where, for each $a\in Q_0$, the
associative $\C$-algebra $\End[\C]{V_a}$ is given its usual Lie
algebra structure.  Note that $\Lie{G}$ has a canonical structure of
an associative unital $\C$-algebra, and that $G$ is the group of units
of the underlying ring of $\Lie{G}$, and is open in $\Lie{G}$.  The
Lie algebra of $H$ is the Lie subalgebra of $\Lie{G}$ consisting of
all elements of the form $ce$, as $c$ runs over $\C$.

\subsection{Moment maps}
\label{sec:moment-maps}

Let $(X,\Omega)$ be a smooth symplectic manifold, and $K$ a real Lie
group.  Suppose we are given a smooth symplectic right action of $K$
on $X$.  A \emph{moment map} for the action of $K$ on $X$ is a smooth
map $\Phi : X \to \Lie{K}^*$, which is $K$-invariant for the coadjoint
action of $K$ on $\Lie{K}^*$, and has the property that
$H(\Phi^\xi)=\xi^\sharp$ for all $\xi \in \Lie{K}$, where $\Phi^\xi$
is the smooth real function $x\mapsto \Phi(x)(\xi)$ on $X$,
$\xi^\sharp$ is the vector field on $X$ induced by $\xi$, and  the
symbol $H(f)$ denotes the Hamiltonian vector field of a smooth real
function $f$ on $X$, which is, by definition, the unique smooth vector
field on $X$, such that $\Omega(x)(H(f)(x),w) = w(f)$ for all $x\in X$
and $w\in \tang[x]{X}$, $\tang[x]{X}$ being the tangent space of $X$
at $x$.

Let $\Omega$ be a symplectic form on a finite-dimensional $\R$-vector
space $V$.  Since the tangent space of $V$ at any point is canonically
isomorphic to $V$ itself, $\Omega$ defines a smooth $2$-form on $V$,
denoted by $\Omega$ itself.  In any linear coordinate system on $V$,
$\Omega$ can be expressed as a form with constant coefficients, so
$\diff \Omega = 0$.  Therefore, $(V,\Omega)$ is a symplectic manifold.

Suppose that we are given a smooth symplectic linear right action of a real lie group $K$ on $V$.  For any element $\xi$ of $\Lie{K}$, let
$\xi^\sharp$ be the vector field on $V$ induced by $\xi$; then
$\xi^\sharp$ is an $\R$-endomorphism of $V$, and
$\Omega(\xi^\sharp(x),y)+\Omega(x,\xi^\sharp(y)) = 0$ for all
$x,y\in V$.  For each element $\alpha$ of $\Lie{K}^*$, define a map
$\Phi_\alpha : V \to \Lie{K}^*$ by
\begin{equation*}
  \label{eq:341}
  \Phi_\alpha(x)(\xi) =
  \frac{1}{2}\Omega(\xi^\sharp(x),x) + \alpha(\xi)
\end{equation*}
for all $x\in V$ and $\xi\in \Lie{K}$.

\begin{lemma} \cite[Lemma 6.1]{PMR}
  \label{lemma1}
  The map $\alpha \mapsto \Phi_\alpha$ is a bijection from the set of
  $K$-invariant elements of $\Lie{K}^*$ onto the set of moment maps
  for the action of $K$ on $V$.
\end{lemma}

We fix a family $h=(h_a)_{a\in Q_0}$ of Hermitian inner products
$h_a : V_a \times V_a \to \C$.  Thus, for every point
$\rho\in \mathcal{A}$, $h$ is a Hermitian metric on the representation
$(V,\rho)$ of $Q$.

For any two finite-dimensional Hermitian inner product spaces $V$ and
$W$, we have a Hermitian inner product $\pair{\cdot}{\cdot}$ on the
$\C$-vector space $\Hom[\C]{V}{W}$, which is defined by
$\pair{u}{v} = \tr{u\circ v^*}$ for all $u,v\in \Hom[\C]{V}{W}$,
where, $v^* : W \to V$ is the adjoint of $v$.  

In particular, the family $h$ induces a Hermitian inner product
$\pair{\cdot}{\cdot}$ on the $\C$-vector space $\Hom[\C]{V_a}{V_b}$
for all $a,b\in Q_0$.  We give $\Lie{G}$ the Hermitian inner product
$\pair{\cdot}{\cdot}$ which is the direct sum of the Hermitian inner
products $\pair{\cdot}{\cdot}$ on $\End[\C]{V_a}$ as $a$ runs over
$Q_0$.  Similarly, we give $\mathcal{A}$ the Hermitian inner product
$\pair{\cdot}{\cdot}$ which is the direct sum of the Hermitian inner
products $\pair{\cdot}{\cdot}$ on
$\Hom[\C]{V_{\sr(\alpha)}}{V_{\tg(\alpha)}}$ as $\alpha$ runs over
$Q_1$.

For each $\rho \in \mathcal{A}$, the $\C$-vector space
$\tang[\rho]{\mathcal{A}}$ is canonically isomorphic to $\mathcal{A}$.
We thus get a Hermitian metric $g$ on the complex manifold
$\mathcal{A}$, defined by $g(\rho)(\sigma,\tau) = \pair{\sigma}{\tau}$
for all $\rho,\sigma,\tau\in \mathcal{A}$, where $\pair{\cdot}{\cdot}$
is the Hermitian inner product on $\mathcal{A}$.  The fundamental
$2$-form $\Omega$ of $g$ is given by
$\Omega(\rho)(\sigma,\tau) = -2\Im(\pair{\sigma}{\tau})$ for all
$\rho,\sigma,\tau\in \mathcal{A}$, where $\Im(t)$ denotes the
imaginary part of a complex number $t$.  Clearly $\diff \Omega = 0$.
Therefore, the Hermitian metric $g$ on $\mathcal{A}$ is K\"ahler.  For
each $a\in Q_0$, let $\Aut{V_a,h_a}$ denote the subgroup of
$\Aut[\C]{V_a}$ consisting of $\C$-automorphisms of $V_a$ which
preserve the Hermitian inner product $h_a$ on $V_a$, and $K$ the
compact subgroup $\prod_{a\in Q_0} \Aut{V_a,h_a}$.  The action of $K$
on $\mathcal{A}$ induced by that of $G$ preserves the Hermitian inner
product $\pair{\cdot}{\cdot}$, and hence the K\"ahler metric $g$ on
$\mathcal{A}$.  Similarly, the action of $K$ on $\Lie{G}$ induced by
that of $G$ preserves the Hermitian inner product on $\Lie{G}$.

The Lie algebra $\Lie{K}$ of $K$ is the real Lie subalgebra
$\bigoplus_{a\in Q_0} \End{V_a,h_a}$ of $\Lie{G}$, where, for each
$a\in Q_0$, $\End{V_a,h_a}$ is the real Lie subalgebra of
$\End[\C]{V_a}$ consisting of $\C$-endomorphisms $u$ of $V_a$ that are
skew-Hermitian with respect to $h_a$, that is,
\begin{equation*}
  \label{eq:324}
  h_a(u(x),y) + h_a(x,u(y)) = 0
\end{equation*}
for all $x,y\in V_a$.
The Hermitian inner product $\pair{\cdot}{\cdot}$ on $\Lie{G}$
restricts to a real inner product on $\Lie{K}$, which is given by
$\pair{\xi}{\eta} = -\sum_{a\in Q_0} \tr{\xi_a \circ \eta_a}$ for all
$\xi,\eta \in \Lie{K}$.

If $a,b\in Q_0$, and $f \in \Hom[\C]{V_a}{V_b}$, let
$f^*\in \Hom[\C]{V_b}{V_a}$ be the adjoint of $f$ with respect to the
Hermitian inner products $h_a$ and $h_b$ on $V_a$ and $V_b$,
respectively.  For every point $\rho$ of $\mathcal{A}$, and
$a\in Q_0$, and a rational weight
$\theta = (\theta_a)_{a \in Q_0} \in \Q^{Q_0}$ define an element
$L_\theta(\rho)_a$ of $\End{V_a,h_a}$ by
\begin{equation*}
  \label{eq:326}
  L_\theta(\rho)_a =
  \sqrt{-1} \Bigl(
  (\theta_a - \mu_\theta(d)) \id{V_a} +
  \sum_{\alpha\in \tg^{-1}(a)} \rho_\alpha \circ \rho_\alpha^* -
  \sum_{\alpha\in \sr^{-1}(a)}\rho_\alpha^* \circ \rho_\alpha
  \Bigr),
\end{equation*}
and let $L_\theta(\rho)$ be the element
$(L_\theta(\rho)_a)_{a\in Q_0}$ of $\Lie{K}$.  Define a map
$\Phi_\theta : \mathcal{A} \to \Lie{K}^*$ by
\begin{equation*}
  \label{eq:335}
  \Phi_\theta(\rho)(\xi) =
  \pair{\xi}{L_\theta(\rho)}
\end{equation*}
for all $\rho\in \mathcal{A}$ and $\xi\in \Lie{K}$, where
$\pair{\cdot}{\cdot}$ is the real inner product on $\Lie{K}$.

\begin{lemma} \cite[Lemma 6.6]{PMR}
  \label{lemma2}
  Let $\eta$ denote the element
  $\bigl(\sqrt{-1} (\theta_a - \mu_\theta(d)) \id{V_a}\bigr)_{a\in
    Q_0}$ of $\Lie{K}$, and $\alpha$ the element of $\Lie{K}^*$, which
  is defined by $\alpha(\xi) = \pair{\xi}{\eta}$ for all
  $\xi\in \Lie{K}$.  Then,
  \begin{equation*}
    \label{eq:343}
    \Phi_\theta(\rho)(\xi) =
    \frac{1}{2}\Omega(\xi^\sharp(\rho),\rho) +
    \alpha(\xi)
  \end{equation*}
  for all $\rho\in \mathcal{A}$ and $\xi\in \Lie{K}$.  In particular,
  $\Phi_\theta$ is a moment map for the action of $K$ on
  $\mathcal{A}$.
\end{lemma}

Let $H = \set{(c \id{V_a})_{a\in Q_0} \suchthat c \in \units{\C}}$ be
the central complex Lie subgroup of $G$, $\bar{G}$ the complex Lie
group $H\backslash G$, and $\pi: G \to \bar{G}$ the canonical
projection.  Let $\bar{K}$ be the compact subgroup $\pi(K)$ of
$\bar{G}$, and $\pi_K : K \to \bar{K}$ the homomorphism of real Lie
groups induced by $\pi$.  The subset $H\cap K$ of $G$ is a real Lie
subgroup of $G$, and $\Lie{H\cap K}$ equals the real Lie subalgebra
$\Lie{H}\cap \Lie{K}$ of $\Lie{G}$.  The map
$\tang[e]{\pi} : \Lie{G} \to \Lie{\bar{G}}$ is a surjective
homomorphism of complex Lie algebras with kernel $\Lie{H}$, and
$\tang[e]{\pi_K} : \Lie{K} \to \Lie{\bar{K}}$ is a surjective
homomorphism of real Lie algebras with kernel $\Lie{H\cap K}$.

The action of $\bar{K}$ on $\mathcal{A}$ induced by that of $\bar{G}$
on $\mathcal{A}$ preserves the K\"ahler metric $g$ on $\mathcal{A}$. 
Also, there exists a unique map
$\bar{\Phi} : \mathcal{A} \to \Lie{\bar{K}}^*$, such that
\begin{equation}
	\label{eq:1}
\Phi(x) = \bar{\Phi}(x) \circ \tang[e]{\pi_K}		\tag{1.1}
\end{equation}
for all $x\in \mathcal{A}$ (see \cite[proof of Corollary 6.4]{PMR}). 
The map $\bar{\Phi}$ is a moment map for the action of $\bar{K}$ on
$\mathcal{A}$.

\section{Properness of Moment maps}
	\label{sec2}
In this section, we give the main result of this note. We show that
the moment map is proper if and only if the quiver does not contain
any cycle.

We retain the notation used in Section \ref{sec1}.  Thus,
$Q = (Q_0, Q_1, s, t)$ is a non-empty finite quiver,
$d=(d_a)_{a\in Q_0}$ a non-zero element of $\N^{Q_0}$, and
$V=(V_a)_{a\in Q_0}$ a fixed family of $\C$-vector spaces, such that
$\dim[\C]{V_a}=d_a$ for all $a\in Q_0$.  The representation space
$\bigoplus_{\alpha\in Q_1}\Hom[\C]{V_{\sr(\alpha)}}{V_{\tg(\alpha)}}$
is denoted by $\mathcal{A}$ while the family
$(h_a)_{a \in Q_0}$ of Hermitian inner products
$h_a : V_a \times V_a \to \C$ is denoted by $h$.  In addition, a
rational weight $\theta \in Q^{Q_0}$ of $Q$ is also fixed.  Also,
recall that $G$ is the complex Lie group
$\prod_{a\in Q_0}\Aut[\C]{V_a}$, with its canonical holomorphic
linear right action on $\mathcal{A}$, and $K$ the compact subgroup
$\prod_{a\in Q_0} \Aut{V_a,h_a}$ of $G$, where, for each $a\in Q_0$,
$\Aut{V_a,h_a}$ is the subgroup of $\Aut[\C]{V_a}$ consisting of
$\C$-automorphisms of $V_a$ which preserve the Hermitian inner
product $h_a$ on $V_a$.

Let $\Phi = \Phi_\theta : \mathcal{A} \to \Lie{K}^*$ be the moment
map for the action of $K$ on $\mathcal{A}$ as given by Lemma
\eqref{lemma2}.

\subsection{A necessary condition for properness}
  \label{sec:2.1}
We start by observing that the moment map $\Phi$, in general, need
not be a proper map.  To see this, let us assume that the quiver $Q$
contains a loop at some vertex; let $\beta \in Q_1$ be a loop at a
vertex $b \in Q_0$. Also assume that $d_b \neq 0$.

For each $n \in \N$, let
$\rho(n) = (\rho_{\alpha}(n))_{\alpha \in Q_1}$ be the representation
of $Q$ defined by
\[ \rho_{\alpha}(n) =
  \begin{cases}
    0       & \quad \text{if } \alpha \neq \beta\\
    n\id{V_b}  & \quad \text{if } \alpha = \beta.
  \end{cases}
\]
We identify $\Lie{K}^*$ with $\Lie{K}$ using the canonical inner
product on $\Lie{K}$.  The element $\Phi_{\theta}(\rho)$ of
$\Lie{K}^*$ can then be identified with the element
$$L_\theta(\rho) =
  \sqrt{-1} \Bigl(
  (\theta_a - \mu_\theta(d)) \id{V_a} +
  \sum_{\alpha\in \tg^{-1}(a)} \rho_\alpha \circ \rho_\alpha^* -
  \sum_{\alpha\in \sr^{-1}(a)}\rho_\alpha^* \circ \rho_\alpha
  \Bigr)_{a \in Q_0}$$
of $\Lie{K}$.

Let us write
$\lambda_a = \theta_a - \mu_\theta(d)$ and
$A(\rho)_a =
 \underset{\alpha\in \tg^{-1}(a)}{\sum} \rho_\alpha \circ \rho_\alpha^* -
  \underset{\alpha\in \sr^{-1}(a)}{\sum}\rho_\alpha^* \circ \rho_\alpha$.
Then
$$L_\theta(\rho)_a =  \sqrt{-1} \left(\lambda_a\id{V_a} + A(\rho)_a\right)$$
for all $a \in Q_0$.  Clearly
$$A(\rho(n))_a = 0$$
for all $a \in Q_0$ and for all $n \in \N$.  Hence
$$\Phi(\rho(n)) = \sqrt{-1}(\lambda_a\id{V_a})_{a \in Q_0}$$
for all $n \in \N$.  Since $\set{\rho(n) \suchthat n \in \N}$ is an
unbounded set in $\mathcal{A}$,
$\Phi^{-1}\left(\sqrt{-1}(\lambda_a\id{V_a})_{a \in Q_0}\right)$
is not compact.  Hence $\Phi$ is not proper.

The above observation can be generalised to the case when the quiver
$Q$ does not contain any cycle.

\begin{proposition}
	\label{prop1}
 Suppose that $Q$ contains a cycle
 $\alpha_1 \alpha_2 \cdots \alpha_l$ of length $l \geq 1$, where
 $\alpha_i : a_{i-1} \to a_i$ for $i = 1, \cdots, l$ and $a_0 = a_l$.
 If $d_{a_i} \neq 0$ for each $i \in \set{1, \cdots, l}$, then the
 moment map $\Phi$ is not proper.
\end{proposition}

\proof For each $n \in \N$, we define an element
$\rho(n) = (\rho_{\alpha}(n))_{\alpha \in Q_1}$ of the representation
space $\mathcal{A}$ of $Q$ by setting
$$\rho_{\alpha}(n) = 0 \quad \text{if } \alpha \notin
\set{\alpha_1, \cdots, \alpha_l},$$
and
\[ (\rho_{\alpha_k}(n))_{ij} =
  \begin{cases}
    n       & \quad \text{if } i=j=1\\
    0       & \text{otherwise}
  \end{cases}
\]
for $1 \leq k \leq l$, $1 \leq i \leq d_{\tg(\alpha_k)}$, $1 \leq j \leq d_{\sr(\alpha_k)}$.  Here we have fixed bases for the vector spaces $V_a$ $(a \in Q_0)$, and identified
$\Hom[\C]{V_{\sr(\alpha)}}{V_{\tg(\alpha)}}$ with the space of
$d_{\tg(\alpha)} \times d_{\sr(\alpha)}$ matrices over the field of
complex numbers.  We then have
$$A(\rho(n))_a = 0$$
for all $a \in Q_0$ and for all $n \in \N$.  Hence
$$\Phi(\rho(n)) = \sqrt{-1}(\lambda_a\id{V_a})_{a \in Q_0}$$
for all $n \in \N$.  Thus, by the argument as before this proposition
$\Phi$ is not a proper map.  \qed

\subsection{Sufficiency of the condition}
  \label{sec:2.2}
In this subsection, we show that the condition of the Proposition
\eqref{prop1} is also sufficient.  Before we prove it, let us see, as
an example, that the moment map is proper when $Q$ is a Kronecker
$n$-quiver.

\begin{remark}
	\label{rem1}
 Let $A$ be an $n \times n$ Hermitian matrix.  Then, by the
 Cauchy-Schwarz inequality,
 $$n\cdot\tr{A^2} \geq (\tr{A})^2.$$
\end{remark}

\begin{lemma}
	\label{lemma3}
 Let $n$ be an integer $\geq 1$, and $Q$ the Kronecker $n$-quiver.
 Then, the moment map $\Phi$ is proper.
\end{lemma}

\proof Let $Q = (Q_0, Q_1, s, t)$, where $Q_0 = \set{a, b}$,
$Q_1 = \set{1, \cdots, n}$, $\sr(j) = a$, $\tg(j) = b$ for all
$j \in Q_1$. Then,
$$\mathcal{A} = \displaystyle\bigoplus_{j = 1}^{n}\Hom[\C]{V_a}{V_b}
\quad \text{and} \quad G = \Aut{V_a} \times \Aut{V_b},$$
and the moment map is given by
\begin{equation*}
\Phi(\rho) = \sqrt{-1} \left( \lambda_a \id{V_a} -
\sum_{j = 1}^{n} {\rho_j}^* \circ \rho,~~
\lambda_b\id{V_b} + \sum_{j = 1}^{n} {\rho_j \circ {\rho_j}^*} \right)
\end{equation*}
for $\rho = (\rho_1, \cdots, \rho_n) \in \mathcal{A}$.  Therefore we get,

\begin{equation*}
\norm{\Phi(\rho)}^2 = \tr{(\sum_{j=1}^{n} \rho_j^* \circ \rho_j)^2} +
\tr{(\sum_{j=1}^{n} \rho_j \circ \rho_j^*)^2} +
2(\lambda_b - \lambda_a)(\sum_{j=1}^{n}\tr{\rho_j^* \circ \rho_j}) +
\lambda_a^2d_a + \lambda_b^2d_b.
\end{equation*}

If either $d_a = 0$ or, $d_b = 0$ then $\mathcal{A} = 0$, and we have
nothing to prove.  So, we assume that $d_a, d_b \neq 0$.  Using Remark
\eqref{rem1}, and the fact that
$\norm{\rho}^2 = \sum_{j=1}^{n}\tr{\rho_j^* \circ \rho_j}
	= \sum_{j=1}^{n}\tr{\rho_j \circ \rho_j^*}$,
we obtain the inequality
\begin{equation*}
\norm{\Phi(\rho)}^2 \geq (1/d_a + 1/d_b)\norm{\rho}^4 +
2(\lambda_b - \lambda_a)\norm{\rho}^2 + \lambda_a^2d_a + \lambda_b^2d_b.
\end{equation*}

Thus we see that every unbounded set in $\mathcal{A}$ is mapped to an
unbounded set in $\Lie{K}$ by the map $\Phi$. Hence $\Phi$ is a
proper map.
 \qed

\begin{lemma}
	\label{lemma5}
 Suppose that $Q$ is arrow-finite, and acyclic.  Then,
 there exists an arrow $\beta \in Q_1$ such that
 $\sr(\beta) \neq \tg(\alpha)$ for all $\alpha \in Q_1$. 
\end{lemma}

\begin{proof}
Let us start with an arbitrary arrow $\alpha \in Q_1$.  Rename
$\alpha$ to $\alpha_1$ and say $\alpha_1 : a_1 \to a_0$.  Since $Q$
has no cycles, in particular no loops,
$\sr(\alpha_1) = a_1 \neq a_0 = \tg(\alpha_1)$.  If
$\sr(\alpha_1) \neq \tg(\alpha)$ for all
$\alpha \in Q_1 \setminus \set{\alpha_1}$ then we are done.  If
$\sr(\alpha_1) = \tg(\alpha)$ for some
$\alpha \in Q_1 \setminus \set{\alpha_1}$, rename $\alpha$ to
$\alpha_2$ and say $\alpha_2 : a_2 \to a_1$.

Now $a_2 = a_1 \Limp \alpha_2$ is a loop, and
$a_2 = a_0 \Limp \alpha_2\alpha_1$ is a cycle of length $2$.
Therefore, by the hypothesis
$\sr(\alpha_2) \neq \tg(\alpha_1),\ \tg(\alpha_2)$.

If $\sr(\alpha_2) \neq \tg(\alpha)$ for all
$\alpha \in Q_1 \setminus \set{\alpha_1, \alpha_2}$ then we are done.
If $\sr(\alpha_1) = \tg(\alpha)$ for some
$\alpha \in Q_1 \setminus \set{\alpha_1, \alpha_2}$, rename $\alpha$
to $\alpha_3$ and say $\alpha_3 : a_3 \to a_2$.

Continuing this process, after a finite steps we arrive at the
desired result because the set $Q_1$ is finite.  For, otherwise, by
repeating the above procedure we get a path of infinite length, which
contradicts the assumption that $Q_1$ is a finite set.
\end{proof}

\begin{proposition}
	\label{prop2}
 Let $Q = (Q_0, Q_1, s, t)$ be an acyclic finite quiver.  Then the
 moment map $\Phi$ is proper.
\end{proposition}

\proof The proof is by induction on the cardinality $\card{Q_1}$ of
$Q_1$.  If $\card{Q_1} = 1$, then the acyclic quiver $Q$ can be
thought of as a Kronecker $1$-quiver with some extra vertices which do
not contribute to the representation space
$\mathcal{A} = \Hom[C]{V_b}{V_c}$ of $Q$.  In fact, if
$Q_1 = \set{\alpha}$, $\alpha : b \to c$, $b,c \in Q_0$, then the
moment map is given by
\[ \Phi(\rho)_a =
  \begin{cases}
   \sqrt{-1}\lambda_a\id{V_a} & \quad \text{if } a \in Q_0 \setminus \set{b, c}\\
   \sqrt{-1} \left( \lambda_b \id{V_b} -
   \rho^* \circ \rho \right)       & \quad \text{if } a = b\\
   \sqrt{-1} \left( \lambda_c \id{V_c} +
   \rho \circ \rho^* \right)      & \quad \text{if } a = c
  \end{cases}
\]
for $\rho \in \mathcal{A}$.  Thus by the method of Lemma\eqref{lemma3}, it follows that $\Phi$ is proper.

Let us assume that the proposition holds if $\card{Q_1} < n$.  We
want to show that the statement is true if $\card{Q_1} = n$. 

Let $\beta \in Q_1$ be a fixed arrow and $\mathcal{B}$ the subspace
$\sum_{\alpha \in Q_1 \setminus \set{\beta}}\Hom[C]{V_{\sr(\alpha)}}
{V_{\tg(\alpha)}}$ of $\mathcal{A}$.  Then, $\mathcal{B}$ is
$K$-invariant, $\restrict{\Omega}{\mathcal{B}}$ is a K\"ahler form on
$\mathcal{B}$, and $\Psi = \restrict{\Phi}{\mathcal{B}}$ is a moment
map for the $K$-action on $\mathcal{B}$.

Let $Q' = (Q_0', Q_1', s', t')$ be the quiver defined by
$Q_0' = Q_0$, $Q_1' = Q_1 \setminus \set{\beta}$, $s' =
\restrict{s}{Q_1'}$, and $t' = \restrict{t}{Q_1'}$.  Then, the
representation space of $Q'$ with dimension vector $d$ is canonically
identified with $\mathcal{B}$ as a K\"ahler $K$-manifold.  Therefore,
by the induction  hypothesis, $\Psi$ is proper.

Now assume that $W$ is a compact set in $\Lie{K}$.  We shall show
that $\Phi^{-1}(W)$ is compact in $\mathcal{A}$.  Since $\mathcal{A}$
and $\Lie{K}$ are finite dimensional normed spaces, there exists
$M \in \R$ such that
\begin{align*}
\norm{\Phi(\rho)} &\leq M \quad \text{for all } \rho \in \Phi^{-1}(W)\\
  \Rightarrow \norm{\Phi(\rho)_a} &\leq M \quad \text{for all }
	\rho \in \Phi^{-1}(W) \quad \text{and for all } a \in Q_0.
\end{align*}
We also have the relation
$$\Phi(\rho)_a = \Psi(\pi(\rho))_a + k(a)(\rho_{\beta}^* \circ \rho_{\beta})$$
for each $\rho \in \mathcal{A}$, where 
\[ k(a) =
  \begin{cases}
    0  & \quad \text{if } a \neq b, c\\
    -1 & \quad \text{if } a = b\\
    1 & \quad \text{if } a = c,
  \end{cases}
\]
and $\pi : \mathcal{A} \to \mathcal{B}$ is the canonical projection. 
Thus $\norm{\Psi(\pi(\rho))_a} \leq M + \norm{\rho_{\beta}}^2$ for
all $\rho \in \Phi^{-1}(W)$ and for all $a \in Q_0$.

Let us choose $\beta$ to be that arrow in $Q$ such that
$\sr(\beta) \neq \tg(\alpha)$ for all $\alpha \in Q_1$, which is
possible by virtue of the Lemma \eqref{lemma5}.  We then have
$$\Phi(\rho)_b = \sqrt{-1} \Bigl( \lambda_b \id{V_b} -
  \sum_{\alpha\in \sr^{-1}(b)}\rho_\alpha^* \circ \rho_\alpha \Bigr).
$$
Therefore, using Remark \eqref{rem1}, we get
$$d_b \norm{\Phi(\rho)_b}^2 \geq \Bigl( \lambda_b d_b -
  \sum_{\alpha\in \sr^{-1}(b)}\norm{\rho_\alpha}^2\Bigr)^2.$$
Thus we see that for $\Phi(\rho)$ to lie inside a compact set it is
necessary that $\rho_\beta$ lies inside a bounded set.  Hence we get
$\Psi(\pi(\rho))_a \leq M_1$ for some $M_1 \in \R$ and for all
$a \in Q_0$, $\rho \in \Phi^{-1}(W)$.

Since $\Psi$ is proper, $\norm{\pi(\rho)} \leq M_2$ for some
$M_2 \in \R$ and for all $\rho \in \Phi^{-1}(W)$.  But
$$\norm{\rho}^2 = \norm{\pi(\rho)}^2 + \norm{\rho_\beta}^2$$
for all $\rho \in \mathcal{A}$.  Therefore $\Phi^{-1}(W)$ is a
bounded set in $\mathcal{A}$ and is hence compact.  This completes
the proof. \qed

\begin{corollary}
 Suppose that $d_a > 0$ for all $a \in Q_0$.  Then, the moment map
 $\bar{\Phi} : \mathcal{A} \to \Lie{\bar{K}}^*$ for the $\bar{K}$
 action on $\mathcal{A}$ is proper if and only if $Q$ is acyclic.\end{corollary}

\proof By equation \eqref{eq:1}, we have
$\Phi(x) = \bar{\Phi}(x) \circ \tang[e]{\pi_K}$ for all
$x \in \mathcal{A}$.  This implies that the diagram
\[
  \begin{tikzcd}
    \mathcal{A} \arrow{r}{\bar{\Phi}} \arrow[swap]{dr}{\Phi} &
    {\Lie{\bar{K}}^*} \arrow{d} {{\tang[e]{\pi_K}}^*} \\
     & {\Lie{K}}^*,
  \end{tikzcd}
\]
where ${\tang[e]{\pi_K}}^*$ is the dual map of $\tang[e]{\pi_K}$, is
commutative.  Since ${\Lie{K}}^*$ and $\Lie{\bar{K}}^*$ are finite
dimensional vector spaces and $\tang[e]{\pi_K}$ is surjective,
${\tang[e]{\pi_K}}^*$ is injective with a closed image and hence is a
proper map.  Since $\Lie{\bar{K}}^*$ is Hausdorff, it follows that
$\bar{\Phi}$ is proper if and only if $\Phi$ is proper.  Therefore,
$\bar{\Phi}$ is proper if and only if $Q$ is acyclic.  \qed

\vspace{.7 cm}

\noindent \textbf{Acknowledgment.}  The author would like to thank his supervisor Prof.\ N.\ Raghavendra for many fruitful discussions. 
He is also grateful to him for going through an earlier draft and
pointing out errors.

\end{document}